\documentclass{ifacconf}

\usepackage{mathtools}
\usepackage{graphicx}      
\usepackage{natbib}        
\usepackage{amsmath,amssymb,amsfonts}
\usepackage{mathrsfs}
\usepackage{graphicx}
\usepackage{textcomp}
\usepackage{xcolor}
 \usepackage{enumerate}

\usepackage{algorithm}
\usepackage{algorithmic}
\newtheorem{assumption}{Assumption}
\newtheorem{theorem}{Theorem}
\newtheorem{lemma}{Lemma}
\newtheorem{remark}{Remark}
\newtheorem{corollary}{Corollary}

\newcommand{\ADD}[1]{\textcolor{black}{{#1}}}

\begin{document}
\begin{frontmatter}

\title{Affine-coupled Distributed  Optimization via Distributed Proximal Jacobian ADMM with Quantized Communication
} 

\thanks{$^\dagger$ Corresponding author}
\thanks[footnoteinfo]{The work of X.D. and A.I.R. was supported by the Guangzhou-HKUST(GZ) Joint Funding Scheme (Grant No. 2025A03J3960). The work of A.I.R. was also supported by the Guangdong Provincial Project (Grant No. 2024QN11G109).}

\author[First]{Xu Du}
\author[First]{Boyu Han}
\author[Second]{Ivano Notarnicola}
\author[Third]{Karl~H.~Johansson} 
\author[First,Forth]{Apostolos I. Rikos$^\dagger$}

\address[First]{The Artificial Intelligence Thrust of the Information Hub, The Hong Kong University of Science and Technology (Guangzhou), Guangzhou, China, (e-mail: \{michaelxudu,boyuhan,apostolosr\}@hkust-gz.edu.cn).}
\address[Second]{Department of Electrical, Electronic, and Information Engineering, University of Bologna, Italy, (e-mail: ivano.notarnicola@unibo.it)}
\address[Third]{The Division of Decision and Control Systems, KTH Royal Institute of Technology, SE-100 44 Stockholm, Sweden, (e-mail: kallej@kth.se)}
\address[Forth]{The Department of Computer Science and Engineering, The Hong Kong University of Science and Technology, Clear Water Bay, Hong Kong, China.}

\begin{abstract}        
This paper investigates distributed resource allocation optimization over directed graphs with limited communication bandwidth. We develop a novel distributed algorithm that integrates the centralized Proximal Jacobian Alternating Direction Method of Multipliers (PJ-ADMM) with a finite-level quantized consensus scheme, enabling nodes to cooperatively solve the optimization {\ADD{in a distributed fashion.}} 
Under the assumption of convex objective functions, we establish that the proposed algorithm achieves sublinear convergence to a neighborhood of the optimal solution, with the convergence accuracy explicitly bounded by the quantization {\ADD{level}}. Numerical experiments validate that the algorithm achieves competitive performance compared to existing approaches while {\ADD{exhibiting}} communication efficiency.
\end{abstract}


\end{frontmatter}

\section{Introduction}

Distributed optimization has gained significant attention with the rise of federated learning \cite{ren2025communication}, robotics and model predictive control (MPC) \cite{stomberg2025decentralized}, power grids \cite{Du2019} and wireless communication \cite{yang2024low}. These applications highlight that, as data volumes grow and optimization variables scale, solving large-scale problems on centralized devices becomes impractical (\cite{2024_doostmohammadian_rikos_Johansson_survey}).

Distributed optimization enables large-scale problems with high-dimensional variables to be partitioned across multiple devices, where each device solves a local subproblem and exchanges information until the global solution is obtained. Broadly, distributed optimization methods fall into two categories \cite{ling2015dlm}: (i) primal decomposition and (ii) dual decomposition. In general, dual decomposition methods exhibit faster convergence and higher accuracy compared to primal approaches \cite{ling2015dlm}. In this work, we focus on dual decomposition, with particular emphasis on the Alternating Direction Method of Multipliers (ADMM) (\cite{boyd2011distributed}) for resource allocation problems.

\noindent
\textbf{Existing Literature.}
ADMM was originally introduced in \cite{Glowinski1975, Gabay1976}. For comprehensive surveys and recent advances in the field, we refer readers to \cite{boyd2011distributed, lin2022alternating}. However, in the context of resource allocation problems, \cite{he2015full} demonstrated that the Jacobian variant of ADMM does not always guarantee convergence. To overcome this limitation, subsequent research has evolved along four primary directions: 
\begin{itemize}
\item \textbf{Dual consensus ADMM}, see, e.g., \cite{chang2016proximal,jiang2022distributed};
\item \textbf{Variable Splitting ADMM}, such as \cite{wang2013solving, notarnicola2022passivity} and  \cite[Algorithm~1]{houska2016augmented};
\item \textbf{Tracking-based ADMM over networks}, e.g., \cite{falsone2020tracking};
\item \textbf{Proximal Jacobian ADMM with step-size modification}, including proximal regularization variants \cite{he2016proximal, he2015full, he2020optimal, yang2022proximal, deng2017parallel,chen2020online, choi2025linear,  shen2022multi,li2018augmented}, and accelerated distributed augmented Lagrangian (ADAL)-type methods \cite{chatzipanagiotis2017convergence,chatzipanagiotis2015augmented}.
\end{itemize}

However, except for \cite{falsone2020tracking}, which considers fully distributed information exchange over networks, the aforementioned approaches suffer from two main limitations: (i) their reliance on centralized coordination mechanisms, which constrains scalability in distributed systems, and (ii) the need for nodes to transmit real-valued messages, imposing significant bandwidth requirements and creating a scalability bottleneck. To the best of our knowledge, for Consensus ADMM-type algorithms, \cite{rikos2023asynchronous} and \cite{Du2025CDCA} leverage the network-based quantized consensus algorithm proposed in \cite{rikos2022non} to design fully distributed quantized schemes. However, extending such analysis to resource allocation type ADMM algorithms remains largely unexplored. The main difficulty lies in the fact that resource allocation problems involve affine coupling constraints rather than simple consensus constraints, which breaks the symmetry required by most quantized consensus protocols and thus prevents the direct application of network-based quantized consensus algorithms.

\textbf{Main Contributions.} 
Motivated by the aforementioned challenges, we propose a novel distributed optimization algorithm based on the centralized Proximal Jacobian Alternating Direction Method of Multipliers (PJ-ADMM) framework \cite{deng2017parallel}, which to the best of our knowledge represents the first integration of PJ-ADMM achieving 
(i) fully distributed operation, (ii) communication over directed graphs, and (iii) quantized information exchange among nodes.
 Drawing inspiration from \cite{rikos2023asynchronous} and \cite{Du2025CDCA}), our approach employs a two-layer architecture that decouples optimization steps from distributed averaging, enabling faster convergence compared to single-layer methods. Specifically, we develop the Quantized distributed Proximal Jacobian ADMM (QDPJ-ADMM) detailed in Algorithm~\ref{alg: distributed Jacobi-Proximal ADMM}, where the inner layer implements the quantized consensus scheme from \cite{rikos2023asynchronous} (Algorithm~\ref{alg:QuAS}) to manage communication through quantized {\ADD{message exchanges}}, while the outer layer handles primal-dual variable updates. Theoretical analysis in Theorem~\ref{theorem} establishes sublinear convergence to a neighborhood of the optimal solution for convex  objectives, with numerical experiments confirming convergence accuracy dependent on {\ADD{the utilized quantization level}} due to exclusive use of quantized data for inter-subproblem communication.

\section{Notation and Preliminaries}\label{sec: Notation}

\textbf{Notation.} 
The symbols $\mathbb{R}$, $\mathbb{Q}$, $\mathbb{Z}$, and $\mathbb{N}$ denote the sets of real, rational, integer, and natural numbers, respectively. Matrices and vectors are denoted by uppercase (e.g., $A$) and lowercase (e.g., $a$) letters, respectively. The transpose of a matrix $A \in \mathbb{R}^{n \times n}$ or a vector $a \in \mathbb{R}^{n}$ is written as $A^\top$ or $a^\top$. For a scalar $a \in \mathbb{R}$, $\lfloor a \rfloor$ and $\lceil a \rceil$ represent the floor and ceiling functions, respectively. These operations apply element-wise to vectors $a \in \mathbb{R}^n$, yielding $\lfloor a \rfloor, \lceil a \rceil \in \mathbb{R}^n$. The vector of all ones is denoted by $\mathbf{1}$, and $I$ is the identity matrix of appropriate dimension. The Euclidean norm of a vector $a$ is written as $|a|$. The value of a variable $x$ at node $i$ and iteration $k$ is denoted by $x_i^{[k]}$.
Furthermore, $|\mathcal{S}|$ denotes the cardinality of a set $\mathcal{S}$ (e.g., $|\mathcal{V}| = N$, as introduced subsequently).
The notation $a \mid b$ indicates that $b \in \mathbb{R}^n$ is the dual variable associated with the constraint $a \in \mathbb{R}^n$.

\textbf{Graph Theory.} 
The communication network is modeled as a directed graph (digraph) $\mathcal{G} = (\mathcal{V}, \mathcal{E})$, where $\mathcal{V} = {1, \dots, N}$ ($N \geq 2$) denotes the set of agents (nodes). The edge set $\mathcal{E} \subseteq \mathcal{V} \times \mathcal{V}$ includes a virtual self-edge $(i,i)$ for every node $i \in \mathcal{V}$. A directed edge from node $i$ to $j$ is denoted $e_{ji} = (j,i) \in \mathcal{E}$. The set of in-neighbors of node $i$, denoted $\mathcal{N}_i^{[k]} = \{ j\in \mathcal{V} \mid (i,j)\in \mathcal{E}\}$, consists of nodes that can transmit information directly to $i$, with its cardinality $\mathcal{D}_i^{[k]} = |\mathcal{N}_i^{[k]}|$ representing the in-degree. Similarly, the out-neighbor set $\mathcal{N}_i^{[k+1]} = \{ l \in \mathcal{V} \mid (l,i) \in \mathcal{E}\}$ contains nodes that directly receive information from $i$, and its cardinality $\mathcal{D}_i^{[k+1]} = |\mathcal{N}_i^{[k+1]}|$ defines the out-degree. The diameter $D$ of $\mathcal{G}$ is the length of the longest shortest directed path between any node pair $i, j \in \mathcal{V}$, and the digraph is strongly connected if a directed path exists between every pair of nodes $i, j \in \mathcal{V}$.

\noindent
\textbf{Quantization.} 
Quantization plays a key role in digital communication networks by lowering bandwidth requirements and increasing communication efficiency. Its representation of signals with a finite number of bits enables the implementation of error-correcting codes such as Reed–Solomon or LDPC, thereby allowing significant gains in robustness against interference \cite{proakis2002communication}. While asymmetric, uniform, and logarithmic quantizers are all well-documented \cite{2019:Wei_Johansson}, the analysis in this paper employs an asymmetric mid-rise quantizer with infinite range. The quantizer is defined as
\begin{equation}\label{eq: quantizer}\small
q^a_{\Delta}(b) \coloneqq \left\lfloor \frac{b}{\Delta} \right\rfloor,
\end{equation}
for input vector $b \in \mathbb{R}^n$ and quantization level $\Delta \in \mathbb{Q}$, with the superscript $a$ indicating its asymmetric nature. It is noted that our results hold for other quantizer types as well.

\section{Problem Formulation and Centralized PJ-ADMM}\label{sec: Problem}

\subsection{Resource Allocation Optimization}\label{sec: Problem formulation}
Consider a communication network represented by a directed graph (digraph) $\mathcal G= (\mathcal V, \mathcal E)$ with  $N = |\mathcal V|$ nodes. 
We consider the following resource allocation problem defined over this network 
\begin{equation}\label{eq: DO}\small
	\begin{split} 
		\min_{x_i\in \mathbb R^{n_i}, \;\; i\in \mathcal V}\;\; &\sum_{i=1}^{N}f_i(x_i)\\
        \mathrm{s.t.}\qquad\;\; &\sum_{i=1}^{N}A_ix_i = b, 
	\end{split}
\end{equation}
where each node $i\in \mathcal V$ is associated with a local objective $f_i: \mathbb R^{n_i} \rightarrow \mathbb R$. The local decision variables $x_i\in \mathbb R^{n_i}$ are affinely coupled through the global equality constraint $\sum_{i=1}^{N}A_ix_i = b\in \mathbb R^m$, where  $A_i \in \mathbb R^{m\times n_i}$ and $b\in \mathbb R^m$.

Now let us assume that communication over the network is subject to limited bandwidth.
In order to solve problem \eqref{eq: DO} while ensuring communication efficiency,
we reformulate it as a quantized resource allocation  optimization problem 
\begin{equation}\label{eq: reformulation}\small
	\begin{split}
\min_{x_i\in \mathbb R^{n_i}, \;\; i\in \mathcal V}\;\; &\sum_{i=1}^{N}f_i(x_i)\\
\mathrm{s.t.}\qquad\;\;
         &\sum_{i=1}^{N}A_ix_i = b,\\
        &\text{nodes communicate with quantized values,}
	\end{split}
\end{equation}
where inter-node communications are restricted to quantized messages due to limited bandwidth.

\subsection{Centralized PJ-ADMM}\label{sec: ADMM}
The Lagrangian of problem \eqref{eq: DO} is defined as
\begin{equation}\label{eq: lagrangian}\small
\begin{split}
    \mathcal L(x,\lambda) \coloneqq& \sum_{i=1}^N f_i(x_i) + \lambda^\top \left( \sum_{i=1}^N A_i x_i -b\right),
\end{split}
\end{equation}
where $x = [x_1^\top, x_2^\top, \cdots, x_N^\top]^\top$ collects all local decision variables, and $\lambda \in \mathbb{R}^m$ denotes the dual multiplier associated with the coupling constraint. The corresponding augmented Lagrangian is given by,
\begin{equation}\label{eq: aug lagrangian}\small
\begin{split}
     \mathcal L_\rho(x,\lambda) \coloneqq& \sum_{i=1}^N f_i(x_i) + \lambda^\top \left( \sum_{i=1}^N A_i x_i -b\right)+ \frac{\rho}{2}\left\|\sum_{i=1}^N A_i x_i -b \right\|^2,
\end{split}
\end{equation}
where $\rho>0$ denotes the penalty parameter.

Focusing on the augmented Lagrangian in \eqref{eq: aug lagrangian}, the centralized PJ-ADMM for solving problem \eqref{eq: DO} was proposed in \cite{deng2017parallel}.
The centralized PJ-ADMM iteration consists of the following two steps:
\begin{equation}\label{eq: Proximal Jacobian ADMM}\small
\left\{
    \begin{aligned}
       & x_i^{[k+1]}=\mathop{\arg\min}_{x_i} f_i(x_i)+ \frac{1}{2} \left\|x_i-x_i^{[k]} \right\|^2_{P_i}\\
       &    \qquad\qquad+\frac{\rho}{2}\left\|A_ix_i+\sum_{j\neq i}^N A_jx_j^{[k]} -b +\frac{\lambda^{[k]}}{\rho} \right\|^2, \quad\forall i\in \mathcal V, \\
      & \lambda^{[k+1]} =  \lambda^{[k]} + \gamma\rho \left( \sum_{i=1}^N A_ix_i^{[k+1]} -b \right),
    \end{aligned}\right.
\end{equation}
where $P_i$ satisfies $P_i\succ \rho \left( \frac{N}{2-\gamma} -1 \right)\left\|A_i\right\|^2$ and $0<\gamma<2$, see \cite[Lemma 2.2]{deng2017parallel}.
In the first step of \eqref{eq: Proximal Jacobian ADMM}, each node minimizes the augmented Lagrangian with an additional proximal regularization term $\frac{1}{2} \left\|x_i-x_i^{[k]} \right\|^2_{P_i}$
with respect to its local variable $x_i$. In the second step, the dual variable $\lambda$ is updated. For convex local objectives  $f_i, \forall i\in \mathcal V$, the PJ-ADMM algorithm admits global sublinear convergence guarantees, as shown in \cite[Section 2]{deng2017parallel}.  

It is important to emphasize that the centralized PJ-ADMM formulation presented in \eqref{eq: Proximal Jacobian ADMM} still depends critically on a centralized coordination mechanism, as the dual variable update necessitates the aggregation of global information from all agents. Such centralized dependency may hinder scalability and robustness in large-scale networked systems. To address this limitation, the next section introduces a fully distributed variant that eliminates the need for a central coordinator. Compared with PJ-ADMM, the proposed algorithm possesses two distinct features:
(i) each node interacts only with its local neighbors while collaboratively solving the global problem, and
(ii) communication efficiency is enhanced through the incorporation of quantized information exchange mechanisms.



\section{Distributed Proximal Jacobian
ADMM with Efficient Communication}\label{sec: algorithm}

In this section, we propose a novel algorithm
to solve problem \eqref{eq: reformulation} in a fully distributed manner. 
Before presenting the algorithmic details, we introduce several standard assumptions that will be essential for the subsequent analysis.

\begin{assumption}\label{ass:1}
The communication network is represented by a \textit{strongly connected} directed graph $\mathcal{G} = (\mathcal{V}, \mathcal{E})$.
Each node $i \in \mathcal{V}$ is assumed to know the network diameter $D$ and a common quantization level $\Delta$.
\end{assumption}

\begin{assumption}\label{ass:2}
For each node $i \in \mathcal{V}$, the local objective function $f_i: \mathbb{R}^n \to \mathbb{R}$ is closed, proper, and convex.
In particular, for all $x_\alpha, x_\beta \in \mathbb{R}^n$, the following inequality holds:
\begin{equation}\label{eq: convex}\small
f_i(x_\alpha) + \partial f_i(x_\alpha)^\top (x_\beta - x_\alpha) \le f_i(x_\beta),
\end{equation}
where $\partial f_i(x_\alpha)$ denotes a sub-gradient of $f_i$ at the point $x_\alpha$. Moreover, the solution set of problem \eqref{eq: reformulation} is assumed to be non-empty.
\end{assumption}

Assumption~\ref{ass:1} provides a necessary condition for the convergence of Algorithm~\ref{alg:QuAS} \cite{rikos2023distributed2}, which serves as the coordination step of Algorithm~\ref{alg: distributed Jacobi-Proximal ADMM} as described in \eqref{eq: global variable}. Moreover, the digraph diameter $D$ can be computed by the nodes in a distributed manner within finite time using a protocol for graph diameter calculation, e.g., \cite{oliva2016distributed}.
Assumption \ref{ass:2} ensures the convexity of each local objective, enabling the establishment of global convergence results for Algorithm \ref{alg: distributed Jacobi-Proximal ADMM}.
Furthermore, Assumption \ref{ass:2} guarantees the existence of an optimal solution for problem \eqref{eq: reformulation}.
 
\subsection{Algorithm Development}\label{sec: algorithm structure}
In this section, we present our proposed distributed
algorithm, detailed below as Algorithm \ref{alg: distributed Jacobi-Proximal ADMM}.

To develop the distributed algorithm, we first introduce an auxiliary variable $d^{[k]}$ defined as
\begin{equation} \label{eq: constraint update}\small
  d^{[k]} \coloneqq \sum_{i=1}^{N} A_i x_i^{[k]} - b.  
\end{equation}
The quantity $d^{[k]}$ represents the global residual of the affine coupled constraints in \eqref{eq: reformulation}, which remains constant in a centralized setting. In the fully distributed implementation, it will be replaced by a locally maintained estimate $\hat d_i^{[k]} $ (see \eqref{eq: global variable}), which is dynamically updated through local communication among neighboring agents.

Algorithm \ref{alg: distributed Jacobi-Proximal ADMM} is structured into three main phases: i) local optimization, ii) coordination among nodes, iii) dual variable update.  In the first phase, each node $i\in\mathcal V$ minimizes its own local augmented Lagrangian function, see \eqref{eq: local uodate}. In the second phase, all nodes exchange the information $\phi_i^{[k+1]} = N(A_i x_i^{[k+1]} - b)$ via the distributed quantized averaging Algorithm~\ref{alg:QuAS}, thereby obtaining the updated auxiliary variable estimation $\hat d_i^{[k+1]}$ as described in \eqref{eq: global variable}. In the third phase, the dual variable $\hat \lambda_i^{[k+1]}$ is updated and subsequently used in the next iteration’s local optimization step, see \eqref{eq: dual of Jacobian ADMM}. Note that  $\hat \lambda_i^{[k+1]}$ can be interpreted as a local copy of the global dual variable $\lambda^{[k+1]}$ in \eqref{eq: Proximal Jacobian ADMM} for each agent $i$. These three phases are repeated iteratively until convergence.

Algorithm~\ref{alg:QuAS}, first proposed in \cite{rikos2022non}, consists of three main components: (i) quantization, (ii) averaging, and (iii) a stopping criterion. Firstly, each node initializes its local information as $\phi_i^{[k+1]}$ and introduces $\chi_i$ to denote the quantized representation of $\phi_i^{[k+1]}$. Secondly, in the averaging step, $\chi_i$ is split into $\xi_i$ equal integer pieces (some pieces may differ by one). Each node retains the piece with the smallest value and transmits the remaining $\xi_i - 1$ pieces to randomly selected out-neighbors $l \in \mathcal N_i^{[k+1]}$ or to itself. Simultaneously, it receives pieces $c_j$ from all in-neighbors $j \in \mathcal N^{[k]}_i$ and updates $\chi_i$ and $\xi_i$ as in \eqref{eq: local information avg}. Thirdly, the algorithm performs max- and min-consensus operations every $D$ time steps. If the difference between the max-consensus $M_i$ and min-consensus $m_i$ satisfies $\left\|M_i-m_i\right\|_{\infty}\leq 1$, then each node $i$
scales its solution according to the quantization level to compute $\hat z_i^{[k+1]}$. Once the stopping condition is satisfied, the solution is scaled accordingly. At this point, Algorithm~\ref{alg:QuAS} terminates, and each node $i$ proceeds to the second phase of Algorithm~\ref{alg: distributed Jacobi-Proximal ADMM}.
The convergence analysis of Algorithm~\ref{alg:QuAS} is provided in \cite[Theorem 1]{rikos2022non}.

\begin{algorithm}[ht]
		\small
	\caption{QDPJ-ADMM: Quantized Distributed Proximal Jacobian ADMM}
		\textbf{Input.} Strongly connected digraph $\mathcal{G} = (\mathcal{V}, \mathcal{E})$, parameter $\rho$, network diameter $D$, quantization level $\Delta$, for each node $i \in \mathcal{V}$. 
    Each node $i \in \mathcal V$ has a local cost function $f_i$. 
    Assumptions~\ref{ass:1} and \ref{ass:2} hold. 
    \\
    \textbf{Initialization.} Randomly chosen dual variable $\hat \lambda_i \in \mathbb{R}^n$, and global variable estimation $\hat d_i \in \mathbb{R}^m$, for each node $i \in \mathcal{V}$. 
    \\
    \textbf{Iteration.}
	\begin{enumerate}
		\item[1.] Paralleled solve local NLP:
		\begin{equation}\label{eq: local uodate}
        \begin{aligned}
        	x_i^{[k+1]}=&\mathop{\arg\min}_{x_i} f_i(x_i)+ \frac{1}{2} \left\|x_i-x_i^{[k]} \right\|^2_{P_i}  \\
            &\quad+\frac{\rho}{2}\left\|A_ix_i-A_ix_i^{[k]}+\hat d_i^{[k]} + \frac{\hat\lambda_i^{[k]}}{\rho} \right\|^2.     
        \end{aligned}
		\end{equation}
		\item[2.] Calculate the global variable estimation executing the distributed following algorithm until convergence:
        \begin{equation}\label{eq: global variable}
            \hat d_i^{[k+1]} = \text{Algorithm \ref{alg:QuAS}}\left(\phi_i^{[k+1]}, D,\Delta\right),
        \end{equation} 
where $\phi_i^{[k+1]} =N(A_i x_i^{[k+1]} -b)$.
		
		\item[3.] Update the dual:
		\begin{equation}\label{eq: dual of Jacobian ADMM}
			\hat \lambda_i^{[k+1]} = \hat \lambda_i^{[k]} + \gamma\rho \hat d_i^{[k+1]}.
		\end{equation}
	\end{enumerate}
    \textbf{Output.} Each node $i$ calculates $x_i^*$ that solves problem \eqref{eq: reformulation}.
	\label{alg: distributed Jacobi-Proximal ADMM}
\end{algorithm}

\begin{algorithm}[ht]
	\small
	\caption{DFQAC: Distributed Finite-time Quantized Average Consensus}
	\textbf{Input.} $\phi_i^{[k+1]}=N(A_i x_i^{[k+1]} -b), D, \Delta$. \\
    \textbf{Initialization.} Each node $i \in \mathcal{V}:$
    \begin{enumerate}
	\item[1.] Assigns probability 
    \begin{equation}
    p_{li}=\left\{
   \begin{aligned}
       &\frac{1}{1+\mathcal D_i^{[k+1]}},\quad& \text{if}\; l\in \mathcal N_i^{[k+1]}\cup \{i\},\\
       &0,\quad & \text{if}\; l\notin \mathcal N_i^{[k+1]}\cup \{i\},
   \end{aligned}
      \right.   
    \end{equation} 
    to each out-neighbor of node $i$.
    \item[2.] Sets $\xi_i = 2$, $\chi_i= 2 q^a_{\Delta}(\phi_i^{[k+1]})$ (see \eqref{eq: quantizer}).
	\end{enumerate}
	\textbf{Iteration.} For time steps $t=1,2,\cdots$ each node $i \in \mathcal V$ does: 
	\begin{enumerate}
	\item[1.] \textbf{If} $t\;\text{mod}(D) = 1$, sets $M_i = \left\lceil \frac{\chi_i}{\xi_i} \right\rceil$ and $m_i=\left\lfloor \frac{\chi_i}{\xi_i} \right\rfloor$.\vspace{2mm}

\item[2.] 
Broadcasts $M_i, m_i$ to each out-neighbor $l \in \mathcal N_i^{[k+1]}$ and receives $M_j, m_j$ from each in-neighbor $j \in \mathcal N_i^{[k]}$. 
Then, sets $M_i = \text{max}_{j \in \mathcal N_i^{[k]} \cup \{i\}}\; M_j$, $m_i = \text{min}_{j \in \mathcal N_i^{[k]}\cup \{i\}}\; m_j$. \vspace{2mm}
\item[3.] Sets $\tau_i = \xi_i$. \vspace{2mm}
\item[4.] \textbf{While} $\tau_i>1$ \textbf{do}
\begin{enumerate}
    \item $c_i=\left\lfloor \frac{\chi_i}{\xi_i} \right\rfloor$. \vspace{2mm}
    \item Sets $\chi_i=\chi_i - c_i$, $\xi_i=\xi_i-1$, $\tau_i=\tau-1$. 
    \item Transmits $c_i$ to randomly chosen out-neighbor $l \in \mathcal N_i^{[k+1]}\cup \{i\}$ with probability $p_{li}$. 
    \item Receives $c_i$ from $j\in \mathcal N_i^{[k]}$ and updates
    \begin{subequations}\label{eq: local information avg}
    \begin{align}
    \chi_i^{[t+1]} &= \chi_i^{[t]} + \sum_{j\in \mathcal N_i^{[k]} }w_{ij}^{[t]} c_j^{[t]}, \\
    \xi_i^{[t+1]} &= \xi_i^{[t]} + \sum_{j\in \mathcal N_i^{[k]}} w_{ij}^{[t]}.
        \end{align}
    \end{subequations}    
    Here $w_{ij}^{[t]} = 1$ if node $i$ receives $c_j^{[t]}$ from node $j$ at step $t$. 
    Otherwise $w_{ij}^{[t]}=0$ and node $i$ does not receive information from node $j$. 
\end{enumerate}
\item[5.] \textbf{if} $t\; \text{mod}\; (D)=0$ and $\left\|M_i-m_i\right\|_{\infty}\leq 1$, set $\hat d_i^{[k+1]} = m_i \Delta$, and stop the operation of the algorithm. 
	\end{enumerate}
    \textbf{Output.} $\hat d_i^{[k+1]}$.
	\label{alg:QuAS}
\end{algorithm}

\subsection{Convergence Analysis}\label{sec: convergence}
In this section, we provide the convergence analysis of
Algorithm \ref{alg: distributed Jacobi-Proximal ADMM}. First, we introduce a lemma that is
important for our analysis. Then, we prove our main result
via a theorem.

\begin{lemma} 
The update of the local estimate $\hat d_i$ for each node $i \in \mathcal{V}$ is given by \eqref{eq: global variable} in Algorithm~\ref{alg: distributed Jacobi-Proximal ADMM}. We note that the operation of Algorithm~\ref{alg:QuAS}, as a consequence of quantized communication, satisfies the inequality,
    \begin{equation} \label{eq: errorbound}\small
    \left\{
    \begin{split} 
        &\hat d_i^{[k+1]} = \frac{1}{N} \sum_{j=1}^N \Delta \left \lfloor \frac{\phi_j^{[k+1]}}{\Delta} \right \rfloor + \kappa_j, \\
		&\left\|\hat d_i^{[k+1]} - d^{[k+1]} \right\|\leq 2\sqrt{m}\Delta,
    \end{split}\right. 
\end{equation}
where $\left\| \kappa_i \right\|_{\infty}\leq \Delta$, $\phi_i^{[k+1]} = N(A_i x_i^{[k+1]} -b), \forall i\in\mathcal V$ and $d^{[k+1]}$ is defined in \eqref{eq: constraint update}.
\end{lemma}
\textit{Proof.}  See \cite[Lemma 1]{rikos2023distributed2}, \cite[Appendix]{jiang2021asynchronous} . \hfill$\blacksquare$



\begin{theorem}\label{theorem}
    Let us consider a digraph $\mathcal G= \left( \mathcal V, \mathcal E\right)$. 
Each node $i \in \mathcal V$ has a local cost function $f_i$, and Assumptions~\ref{ass:1} and \ref{ass:2} hold. 
Each node $i \in \mathcal V$ in the network executes Algorithm~\ref{alg: distributed Jacobi-Proximal ADMM} for solving the resource allocation optimization problem in~\eqref{eq: reformulation} in a distributed fashion. 
Given parameter $\rho>0$, 
during the operation of Algorithm~\ref{alg: distributed Jacobi-Proximal ADMM} there always exists a $0<\gamma<2$ and $P_i, \forall i \in \mathcal{V}$ such that 
\begin{equation}\label{eq: condition}\small
    P_i\succ \rho \left( \frac{N}{2-\gamma} -1 \right)\left\|A_i\right\|^2.
\end{equation}
From \eqref{eq: condition}, we have that during the operation of Algorithm~\ref{alg: distributed Jacobi-Proximal ADMM} the following inequality is satisfied 
\begin{equation}\label{eq: result}\small
\begin{split}
 &\mathcal L\left(\frac{\sum_{k=1}^K x^{[k]}}{K},\lambda^*\right) - \mathcal L(x^*,\lambda^*) \\
 \leq	&	\frac{1}{K} \sum_{k=1}^{K} \left( \mathcal L\left(x^{[k]},\lambda^*\right) - \mathcal L(x^*,\lambda^*)\right)\\
 \leq &\frac{C}{K}  + \mathcal O(\Delta),
\end{split}
\end{equation}
where 
\begin{equation}\label{eq: upper bound}\small 
    \begin{split}
        C \coloneqq      \frac{1}{2}\sum_{i=1}^{N}\left\|x_i^{[1]}-x_i^* \right\|^2_{\rho A_i^\top A_i + P_i} + \frac{1}{2\gamma \rho}\left\|\hat \lambda_i^{[1]}-\lambda^* \right\|^2,
    \end{split}
\end{equation}
$\Delta$ denotes the quantization level and $(x^*, \lambda^*)$ is a saddle point of \eqref{eq: reformulation}.
 \end{theorem}

\textit{Proof.} See Appendix \ref{appendix: theorem 1}.  \hfill$\blacksquare$

The convergence result in Theorem~\ref{theorem} is inherited from \cite{deng2017parallel}. Based on this theorem, we derive a convergence result for the decision variables $x_i$, as stated in the following corollary.

\begin{corollary}\label{corollary 1}
Under the same conditions as in Theorem~\ref{theorem}, it holds that
\begin{equation}\label{eq: accuracy}\small
\begin{split}
&\lim_{K \rightarrow \infty} \frac{1}{K} \sum_{k=1}^K 
\frac{1}{2} \sum_{i=1}^{N} 
\left\| x_i^{[k]} - x_i^{[k+1]} \right\|^2_{\rho A_i^\top A_i + P_i - \frac{\rho}{\epsilon_i} A_i^\top A_i} \\
\leq &\ \frac{C}{K} + \mathcal O(\Delta),
\end{split}
\end{equation}
where the constant $C$ is defined in \eqref{eq: upper bound}, and each $\epsilon_i > 0$ for $i \in \mathcal{V}$, as shown in Appendix \ref{appendix: theorem 1}.
\end{corollary}
\textit{Proof.}  See Appendix \ref{Appendix corollary 1}.
 \hfill$\blacksquare$


\begin{remark}
    Under the convexity of the objective functions $f_i$, Algorithm \ref{alg: distributed Jacobi-Proximal ADMM} guarantees that the function values converge to a neighborhood of the optimal value, with the size of this neighborhood ultimately determined by the quantization level $\Delta$. Since $C$ is a constant, the term $\frac{C}{K}$ tends to zero as $K\rightarrow \infty$. This implies that the left-hand side of \eqref{eq: accuracy} is upper bounded by $\mathcal O(\Delta)$. A more rigorous analysis of variable convergence, including the convergence of $\left\|x_i^{[k]} - x_i^*\right\|^2$ and the feasibility of coupling constraints in \eqref{eq: reformulation}, will be addressed in an extended version of this paper.
\end{remark}

\section{Numerical Simulation}\label{sec: numerical}
In this section, we present numerical simulations to illustrate the performance of Algorithm~\ref{alg: distributed Jacobi-Proximal ADMM} and highlight its improvements over existing optimization methods.

Our experiments are based on the setup in \cite[Section 3.7]{deng2017parallel}, with implementation details available at \texttt{https://github.com/ZhiminPeng/Jacobi-ADMM}.
We consider the following resource allocation problem:
\begin{equation}\label{eq: problem}\small \begin{split} \min_{x_i\in \mathbb R^{n_i}, \;\; i\in \mathcal V}\;\; \sum_{i=1}^{N}\frac{1}{2}\|C_i x_i -e_i\|^2\quad \quad\mathrm{s.t.}\;\; \sum_{i=1}^{N}x_i = 0, \end{split} \end{equation}
defined over a random digraph with $N=100$ nodes, where $x_i \in \mathbb{R}^{100}$. Here $C_i \in \mathbb{R}^{120\times 100}$, and $e_i \in \mathbb{R}^{120}$ for all $i \in \mathcal V$ are randomly generated. For further details, please refer to \cite[Section 3.7]{deng2017parallel}.
 We set $\rho = 0.01$, $\gamma = 1$, and $P_i = \tau I \succ 0$ with $\tau = \rho (N -1+ 10^{-3})$ to satisfy the condition \eqref{eq: condition}.
Algorithm~\ref{alg: distributed Jacobi-Proximal ADMM} is executed for quantization levels $\Delta = 10^{-3}, 10^{-4}, 10^{-5}, 10^{-6}$. We evaluate convergence by plotting the error $\sum_{i=1}^N \|x_i^{[k]} - x_i^*\|_1,\; \forall i\in \mathcal V,$ where $x^* = [(x_1^*)^\top, (x_2^*)^\top, \dots, (x_N^*)^\top]^\top$ denotes the optimal solution of \eqref{eq: reformulation}.

\begin{figure}[ht]
	\centering
\includegraphics[width=0.5\textwidth,height=0.27\textheight]{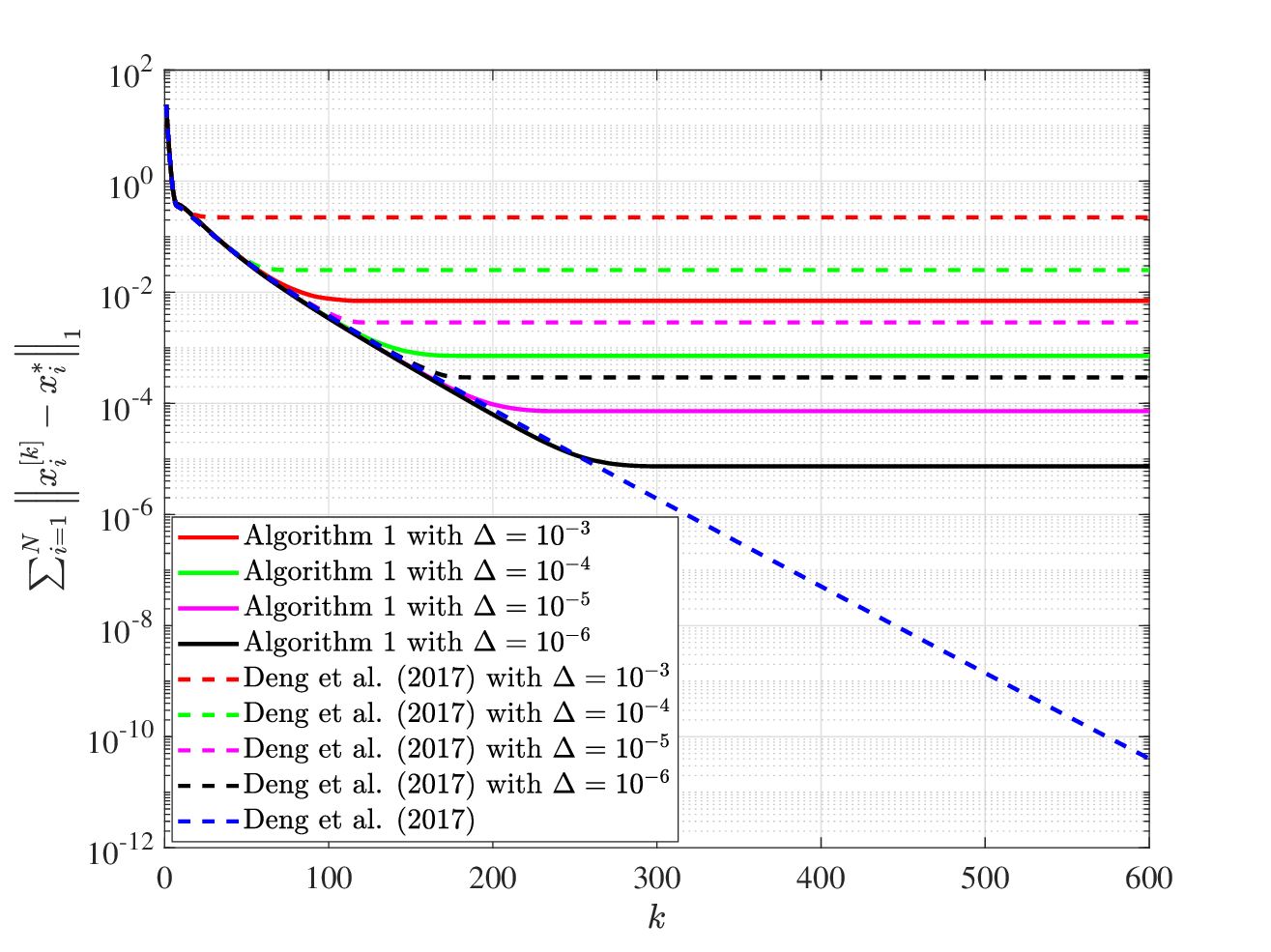}
	\caption{Comparison of Algorithm \ref{alg: distributed Jacobi-Proximal ADMM} 
    with PJ-ADMM (\cite{deng2017parallel}) over a directed graph with quantization level $\Delta = 10^{-3}, 10^{-4}, 10^{-5}$  and $10^{-6}$.}
	\label{fig: comparison}
\end{figure}


From Fig.~\ref{fig: comparison}, it can be observed that Algorithm~\ref{alg: distributed Jacobi-Proximal ADMM} converges to a neighborhood of the optimal solution of \eqref{eq: reformulation}. In particular, a smaller quantization level $\Delta$ yields higher accuracy, which is consistent with the theoretical results in Theorem~\ref{theorem}. This demonstrates that Algorithm~\ref{alg: distributed Jacobi-Proximal ADMM} achieves a tradeoff between communication efficiency and solution accuracy, as a larger $\Delta$ reduces communication bandwidth requirements among nodes. Moreover, the algorithm operates in a fully distributed manner, without relying on a central server.

For comparison, the dashed lines in Fig.~\ref{fig: comparison} show the convergence of the centralized PJ-ADMM~\eqref{eq: Proximal Jacobian ADMM} under different quantization levels $\Delta$, where the last dashed line corresponds to real-valued transmission. Since each iteration of~\eqref{eq: Proximal Jacobian ADMM} requires exchanging both $x_i^{[k+1]}$ and $\lambda^{[k+1]}$, the centralized scheme accumulates larger quantization errors than Algorithm~\ref{alg: distributed Jacobi-Proximal ADMM}.

\section{Conclusion}


In this work, we proposed DQPJ-ADMM, a novel distributed algorithm that computes optimal solutions over directed networks under quantized communication. Operating without a central coordinator, it achieves sublinear convergence to a neighborhood of the optimum for convex resource allocation problems, with accuracy dependent on the quantization level. Theoretical analysis ensures convergence, while numerical tests validate the results and demonstrate improved communication efficiency over existing baselines. Future work will include a rigorous convergence analysis quantifying how the presence of quantization noise affects the accuracy of the final optimal solution under stronger assumptions, along with extending the proposed method to open network environments with dynamic topologies and agent participation.

\appendix
\section{Proof of Theorem \ref{theorem}}\label{appendix: theorem 1}
The proof utilizes the following known results. First, the identity:
\begin{equation}\label{eq: condition3}\small
(a - b)^{\top} (b - c)
= \frac{1}{2} \big( -\|b - c\|^2 - \|a - b\|^2 + \|a - c\|^2 \big).
\end{equation}
Second, we recall Young’s inequality. For any vectors $a, b \in \mathbb{R}^n$ and a constant $\epsilon > 0$, the following holds:
\begin{equation}\label{eq: young}\small
a^\top b \leq \frac{\epsilon}{2} \|a\|^2 + \frac{1}{2\epsilon} \|b\|^2.
\end{equation}
Next, we establish a pivotal relationship for the dual variables in Algorithm~\ref{alg: distributed Jacobi-Proximal ADMM} applied to \eqref{eq: reformulation}. For the convenience of subsequent analysis, we introduce the intermediate dual variable $\tilde{\lambda}_i^{[k+1]} \coloneqq \hat{\lambda}_i^{[k]} + \rho \hat{d}_i^{[k+1]}$. This variable connects to the relaxed dual update $\hat{\lambda}_i$ in \eqref{eq: dual of Jacobian ADMM} through the relaxation parameter $\gamma$, as demonstrated by the following derivation:
\begin{equation}\label{eq: dual jacobian ADMM2}\small
	\begin{split}
		&\tilde{\lambda}_i^{[k+1]} - \lambda^*\\=&\left(\tilde{\lambda}_i^{[k+1]} - \hat \lambda_i^{[k+1]}  \right) + \left(\hat{\lambda}_i^{[k+1]} -  \lambda^*  \right)\\
		=& \left(\hat \lambda_i^{[k]}+\rho \hat d_i^{[k+1]} -  \hat \lambda_i^{[k]} - \gamma\rho \hat d_i^{[k+1]}  \right) + \left(\hat{\lambda}_i^{[k+1]} -  \lambda^* \right)\\
		\overset{\eqref{eq: dual of Jacobian ADMM}}{=}& \frac{\gamma-1}{\gamma}\left(\hat \lambda_i^{[k]} -\hat\lambda_i^{[k+1]} \right) + \left( \hat \lambda_i^{[k+1]} -\lambda^*\right).
	\end{split}
\end{equation}

We introduce an auxiliary function \cite[Appendix, Proof of inequality (A.2)]{boyd2011distributed}:
\begin{equation}\small
\begin{split}
F_i(\xi) \coloneqq\ & f_i(\xi) 
+ \Big(A_i^\top \hat \lambda_i^{[k]} 
+ \rho A_i^\top \big(A_i(x_i^{[k+1]} - x_i^{[k]}) + \hat d_i^{[k]} \big) \\
& + P_i (x_i^{[k+1]} - x_i^{[k]}) \Big)^\top \xi.
\end{split}
\end{equation}

By the optimality condition of convex functions, we have $\sum_{i=1}^N F_i(x_i^{[k+1]}) - F_i(x_i^*) \overset{\eqref{eq: convex}}{\leq} 0,$
which implies
\begin{equation}\label{eq: begin}\small
\begin{split}
&\sum_{i=1}^N f_i\left(x_i^{[k+1]}\right) - f_i(x_i^*) \\\leq & 
\sum_{i=1}^{N} \left(x_i^* - x_i^{[k+1]}\right)^\top \left(A_i^\top \hat \lambda_i^{[k]} + \hat d_i^{[k]}\right) \\
& + \sum_{i=1}^{N} \left(x_i^* - x_i^{[k+1]}\right)^\top \rho A_i^\top A_i \left(x_i^{[k+1]} - x_i^{[k]}\right) \\
& + \sum_{i=1}^{N} \left(x_i^* - x_i^{[k+1]}\right)^\top P_i \left(x_i^{[k+1]} - x_i^{[k]}\right).
\end{split}
\end{equation}

Adding $(\lambda^*)^\top \sum_{i=1}^{N} A_i (x_i^{[k+1]} - x_i^*)$ to both sides and using \eqref{eq: condition3}, \eqref{eq: begin} can be compactly rewritten as
\begin{equation}\label{eq: center}\small
\begin{split}
0 \leq &\ \mathcal L\left(x^{[k+1]},\lambda^*\right) - \mathcal L(x^*,\lambda^*) \\
\leq &\ \Phi\left(x^{[k+1]},x^{[k]},x^*\right) \\
& + \left(\sum_{i=1}^{N} A_i \left(x_i^* - x_i^{[k+1]}\right) \right)^\top 
      \left (\hat \lambda_i^{[k]} - \lambda^* + \rho \hat d_i^{[k]}\right),
\end{split}
\end{equation}
where
\begin{equation}\label{eq: primal equation}\small
\begin{split}
&\Phi\left(x^{[k+1]},x^{[k]},x^*\right)\\ \coloneqq\ & 
\frac{1}{2} \sum_{i=1}^{N} \left\|x_i^{[k]} - x_i^*\right\|^2_{\rho A_i^\top A_i + P_i}  - \frac{1}{2} \sum_{i=1}^{N} \left\|x_i^{[k+1]} - x_i^*\right\|^2_{\rho A_i^\top A_i + P_i} \\
& - \frac{1}{2} \sum_{i=1}^{N} \left\|x_i^{[k+1]} - x_i^{[k]}\right\|^2_{\rho A_i^\top A_i + P_i}.
\end{split}
\end{equation}
Note that the non-negativity in the first line of \eqref{eq: center} follows directly from the saddle point theorem.

The last term in \eqref{eq: center} can be further bounded as
\begin{equation}\label{eq: decouple}\small
\begin{split}
& \left(\sum_{i=1}^{N} A_i (x_i^* - x_i^{[k+1]}) \right)^\top 
  \big(\hat \lambda_i^{[k]} - \lambda^* + \rho \hat d_i^{[k]}\big) \\
\overset{\eqref{eq: dual of Jacobian ADMM},\eqref{eq: errorbound},\eqref{eq: dual jacobian ADMM2}}{\leq} & 
\underbrace{\left( \frac{\hat \lambda_i^{[k]} - \hat \lambda_i^{[k+1]}}{\gamma \rho} \right)^\top 
\big(\tilde \lambda_i^{[k+1]} - \lambda^* \big)}_{\text{(a)}} \\
& + \underbrace{\rho \left( \frac{\hat \lambda_i^{[k]} - \hat \lambda_i^{[k+1]}}{\gamma \rho} \right)^\top 
\big(\hat d_i^{[k]} - \hat d_i^{[k+1]} \big)}_{\text{(b)}} \\
& + \underbrace{2\sqrt{m}\Delta \left\| \tilde \lambda_i^{[k+1]} - \lambda^* + \rho \left(\hat d_i^{[k]} - \hat d_i^{[k+1]}\right) \right\|}_{\text{(c)}}.
\end{split}
\end{equation}

We now decompose the right-hand side of \eqref{eq: decouple} into three components, denoted as (a), (b), and (c), and analyze each component separately.

First, component (a) can be expressed as
\begin{equation}\label{eq: a}\small
\begin{split}
& \left( \frac{\hat \lambda_i^{[k]} - \hat \lambda_i^{[k+1]}}{\gamma \rho} \right)^\top 
\left(\tilde \lambda_i^{[k+1]} - \lambda^*\right) \\
\overset{\eqref{eq: dual jacobian ADMM2}}{=} &
\left( \frac{\hat \lambda_i^{[k]} - \hat \lambda_i^{[k+1]}}{\gamma \rho} \right)^\top\\
&\left( \frac{\gamma - 1}{\gamma} \left(\hat \lambda_i^{[k]} - \hat \lambda_i^{[k+1]}\right)
+ \left(\hat \lambda_i^{[k+1]} - \lambda^*\right) \right) \\
\overset{\eqref{eq: condition3}}{=} & \frac{\gamma - 1}{\gamma^2 \rho} 
\left\|\hat \lambda_i^{[k]} - \hat \lambda_i^{[k+1]}\right\|^2 
+ \Psi\left(\hat \lambda^{[k+1]}_i, \hat \lambda^{[k]}_i, \lambda^*\right),
\end{split}
\end{equation}
where
\begin{equation}\label{eq: dual equation}\small
\begin{split}
\Psi\left(\hat \lambda^{[k+1]}_i, \hat \lambda^{[k]}_i, \lambda^*\right) 
\coloneqq &\ \frac{1}{2\gamma \rho} \Big(
\left\|\hat \lambda_i^{[k]} - \lambda^*\right\|^2 
- \left\|\hat \lambda_i^{[k+1]} - \hat \lambda_i^{[k]}\right\|^2 \\
&\quad - \left\|\hat \lambda_i^{[k+1]} - \lambda^*\right\|^2 \Big).
\end{split}
\end{equation}

Second, component (b) can be written as
\begin{equation}\label{eq: b1}\small
\begin{split}
& \left( \frac{\hat \lambda_i^{[k]} - \hat \lambda_i^{[k+1]}}{\gamma} \right)^\top 
\left(\hat d_i^{[k]} - \hat d_i^{[k+1]}\right) \\
\overset{\eqref{eq: constraint update}, \eqref{eq: errorbound}}{\leq} &
\left( \frac{\hat \lambda_i^{[k]} - \hat \lambda_i^{[k+1]}}{\gamma} \right)^\top
\left( \sum_{i=1}^{N} A_i (x_i^{[k]} - x_i^{[k+1]}) \right) \\
&\quad + \frac{4\sqrt{m}\Delta}{\gamma} 
\left\|\hat \lambda_i^{[k]} - \hat \lambda_i^{[k+1]}\right\|.
\end{split}
\end{equation}
By applying \eqref{eq: young}, the right hand side of inequality  \eqref{eq: b1} can be further bounded as
\begin{equation}\label{eq: b}\small
\begin{split}
 &
\frac{1}{2} \Big( \sum_{i=1}^{N} \epsilon_i \Big)
\frac{1}{\gamma^2 \rho} 
\left\|\hat \lambda_i^{[k]} - \hat \lambda_i^{[k+1]}\right\|^2 \\
& + \frac{1}{2} \sum_{i=1}^{N} 
\frac{\rho}{\epsilon_i} 
\left\|x_i^{[k]} - x_i^{[k+1]}\right\|^2_{A_i^\top A_i} \\
& + \frac{4\sqrt{m}\Delta}{\gamma} 
\left(\left \|\hat \lambda_i^{[k]} - \lambda^*\right\| 
+ \left\|\hat \lambda_i^{[k+1]} - \lambda^*\right\| \right).
\end{split}
\end{equation}

Finally, component (c) can be expressed as
\begin{equation}\label{eq: c}\small
\begin{split}
& 2\sqrt{m}\Delta 
\left\| \tilde \lambda_i^{[k+1]} - \lambda^* 
+ \rho (\hat d_i^{[k]} - \hat d_i^{[k+1]}) \right\| \\
\overset{\eqref{eq: dual jacobian ADMM2}}{\leq} & 2\sqrt{m}\Delta \frac{|\gamma-1|}{\gamma}\left (\left\|\hat \lambda_i^{[k]} -\lambda^* \right\|+ \left\|\hat\lambda_i^{[k+1]} - \lambda^*\right\| \right) \\
&\hspace{-5mm}+ 2\sqrt{m}\Delta\left\| \hat \lambda_i^{[k+1]} -\lambda^*\right\|+2\rho \sqrt{m}\Delta  \left(\left\|  \hat d_i^{[k]}\right\| + \left\|\hat d_i^{[k+1]} \right\|\right). 
\end{split}
\end{equation}

As a summary, by summing up \eqref{eq: a},\eqref{eq: b} and \eqref{eq: c}, the last term in \eqref{eq: center} can be bounded as
\begin{equation}\label{eq: splitting sum}\small
\begin{split}
& \left(\sum_{i=1}^{N} A_i \left(x_i^* - x_i^{[k+1]}\right) \right)^\top
\left(\hat \lambda_i^{[k]} - \lambda^* + \rho \hat d_i^{[k]}\right)\\
&\hspace{-7mm}\overset{\eqref{eq: a},\eqref{eq: b},\eqref{eq: c}}{\leq} \Gamma^{[k]},
\end{split}
\end{equation}
where $\Gamma^{[k]}$ denotes the summation of the right-hand sides of \eqref{eq: a},\eqref{eq: b} and \eqref{eq: c}.

Moreover, according to \cite[Lemma~2.1]{deng2017parallel}, the sequence $\sum_{i=1}^N \|x_i^{[k]} - x_i^* \|_{P_i+\rho A_i^\top A_i}^2 
+ \frac{1}{\gamma \rho}\|\lambda^{[k]} - \lambda^*\|^2$
is monotonically decreasing under \eqref{eq: Proximal Jacobian ADMM}. 
Since the initial values $\lambda_i^{[1]}$, $x_i^{[1]}$, and the optimal solutions $x_i^*$ for all $i \in \mathcal V$ are bounded, and the quantization level $\Delta$ is fixed and finite, both $\|\hat d_i^{[k]}\|$ and $\|\hat\lambda_i^{[k]} - \lambda^*\|$ remain bounded.  
Hence, there exist constants $M_x>0$ and $M_\lambda>0$ such that $\|\hat d_i^{[k]}\| \leq M_x, \quad \|\hat\lambda_i^{[k]} - \lambda^*\| \leq M_\lambda,$
as similarly stated in \cite[Page~8]{jiang2021asynchronous}.  

To quantify the impact of quantization in \eqref{eq: b} and \eqref{eq: c}, we define
\begin{equation}\label{eq: error bound}\small
\begin{split}
    \mathcal O(\Delta) \coloneqq 
\left(\frac{8\Delta}{\gamma} + 4\Delta \frac{|\gamma-1|}{\gamma}+ 2\Delta \right)\sqrt{m} M_\lambda 
+ 4 \rho \sqrt{m}\Delta M_x.
\end{split}
\end{equation}

Substituting \eqref{eq: splitting sum} and \eqref{eq: error bound} into \eqref{eq: center} yields  
\begin{equation}\label{eq: relaxed0}\small
	\begin{split}
		&0 \leq \mathcal L\left(x^{[k+1]},\lambda^*\right) - \mathcal L(x^*,\lambda^*)\\
		\overset{\eqref{eq: center}, \eqref{eq: splitting sum} ,\eqref{eq: error bound}}{\leq} &
        \left\{ \sum_{i=1}^N \left\|x_i^{[k]} - x_i^* \right\|_{P_i+\rho A_i^\top A_i}^2 
+ \frac{1}{\gamma \rho}\left\|\hat\lambda^{[k]}_i - \lambda^*\right\|^2\right\}\\
 &\hspace{-15mm}- \left\{\sum_{i=1}^N \left\|x_i^{[k+1]} - x_i^* \right\|_{P_i+\rho A_i^\top A_i}^2 
+ \frac{1}{\gamma \rho}\left\|\hat\lambda^{[k+1]}_i - \lambda^*\right\|^2\right\}\\
 &\hspace{-15mm}- \left\{\sum_{i=1}^N\left \|x_i^{[k+1]} - x_i^{[k]} \right\|_{P_i+\rho A_i^\top A_i}^2 +\frac{1}{\gamma \rho}\left\|\hat\lambda^{[k+1]}_i - \hat\lambda^{[k]}_i\right\|^2\right\}\\
		&\hspace{-15mm} + \left( \frac{1}{2}\left(\sum_{i=1}^{N}\!\epsilon_i\right)\frac{1}{\gamma^2\rho} + \frac{\gamma-1}{\gamma^2\rho}\right)
   \left \|\hat \lambda_i^{[k]} - \hat \lambda_i^{[k+1]} \right\|^2 
     \\
        &\hspace{-15mm} + \frac{1}{2}\sum_{i=1}^{N} 
        \frac{\rho}{\epsilon_i} 
      \left  \|x_i^{[k]}-x_i^{[k+1]}\right\|^2_{A_i^\top A_i}+ \mathcal O (\Delta). 
	\end{split}
\end{equation}

By collecting  the terms involving 
$\left\|\hat \lambda_i^{[k]}-\hat \lambda_i^{[k+1]}\right\|^2$  
and $\left\|x_i^{[k]}-x_i^{[k+1]}\right\|^2$ in \eqref{eq: relaxed0}, we define the auxiliary term  
\begin{equation}\label{eq: auxi}\small
    \begin{split}
        &\Theta(\psi^{[k+1]},\psi^{[k]})\\
        \coloneqq&
        \left(\frac{1}{2\gamma\rho} - \frac{\gamma-1}{\rho \gamma^2} 
        - \frac{1}{2}\left(\sum_{i=1}^{N}\epsilon_i\right)\frac{1}{\gamma^2\rho}\right)
       \left \|\hat \lambda_i^{[k]}-\hat \lambda_i^{[k+1]}\right \|^2\\
        &\quad + 
        \frac{1}{2}\sum_{i=1}^{N} 
      \left  \|x_i^{[k]} - x_i^{[k+1]}\right \|^2_{\rho A_i^\top A_i 
        + P_i - \frac{\rho}{\epsilon_i}A_i^\top A_i}.
    \end{split}
\end{equation}
Since condition \eqref{eq: condition} holds, it follows that 
$\Theta(\psi^{[k+1]},\psi^{[k]}) \geq 0$.  
This result is analogous to that in \cite[Lemma~2.2]{deng2017parallel}.

Finally, substituting \eqref{eq: auxi} into \eqref{eq: relaxed0} yields
\begin{equation} \label{eq: decreasing}\small
	\begin{split}
		0 &\leq \mathcal L\left(x^{[k+1]},\lambda^*\right) - \mathcal L(x^*,\lambda^*)\\
		&\overset{\eqref{eq: auxi}}{\leq}
		\Bigg\{
    \frac{1}{2}\sum_{i=1}^{N}
   \left \|x_i^{[k]}-x_i^* \right\|^2_{\rho A_i^\top A_i + P_i} 
    + \frac{1}{2\gamma \rho}\left\|\hat \lambda_i^{[k]}-\lambda^* \right\|^2 
    \Bigg\}\\
         -&
    \Bigg\{
    \frac{1}{2}\sum_{i=1}^{N}
    \left\|x_i^{[k+1]}-x_i^* \right\|^2_{\rho A_i^\top A_i + P_i} 
    + \frac{1}{2\gamma \rho}\left\|\hat \lambda_i^{[k+1]}-\lambda^* \right\|^2 
    \Bigg\}\\
		&\quad - \Theta\left(\psi^{[k+1]},\psi^{[k]}\right) + \mathcal O(\Delta).
 	\end{split}
\end{equation}
Under Assumption~\ref{ass:2}, applying a telescoping summation to \eqref{eq: decreasing} directly yields \eqref{eq: result}.
 \hfill$\blacksquare$

\section{Proof of Corollary \ref{corollary 1}}\label{Appendix corollary 1}
Summing \eqref{eq: decreasing} over $k=1,\dots,K$ yields
\begin{equation}\label{eq: sumTheta}\small
\begin{split}
        \frac{1}{K} \sum_{k=1}^K \Theta\left(\psi^{[k+1]},\psi^{[k]}\right) 
    \leq \frac{C}{K} + \mathcal O(\Delta).
\end{split}
\end{equation}

Taking the limit as $K \rightarrow \infty$ and noting that
\begin{equation}\small
\begin{split}
     \frac{1}{2} \sum_{i=1}^{N}\left \|x_i^{[k]} - x_i^{[k+1]}\right\|^2_{\rho A_i^\top A_i + P_i - \frac{\rho}{\epsilon_i} A_i^\top A_i} 
    \leq \Theta\left(\psi^{[k+1]},\psi^{[k]}\right),
\end{split}
\end{equation}
from \eqref{eq: auxi}, the desired result \eqref{eq: accuracy} follows.

\bibliography{ifacconf}             
\end{document}